\documentclass[12pt, twoside]{amsart} 
 \title{Reduction maps and minimal model theory}
\author{Yoshinori  Gongyo}

\address{Graduate School of Mathematical Sciences, 
the University of Tokyo, 3-8-1 Komaba, Meguro-ku, Tokyo 153-8914, Japan.}
\email{gongyo@ms.u-tokyo.ac.jp}
\author{Brian Lehmann}

\address{Department of Mathematics, Rice University,
Houston, TX 77005}
\email{blehmann@rice.edu}
\date{2012/3/30, version 2.07}

\newcommand{\Supp}[0]{{\operatorname{Supp}}}

\newcommand{\mult}[0]{{\operatorname{mult}}}

\usepackage{latexsym}
\usepackage{amsmath}
\usepackage{amssymb}
\usepackage{amsthm}
\usepackage{amscd}
\usepackage{enumerate} 
\usepackage{amssymb} 
\usepackage{mathrsfs}    
\usepackage[all]{xy}

\usepackage{latexsym}

\newtheorem{thm}{Theorem}[section]

\newtheorem{prop}[thm]{Proposition}

\newtheorem{lem}[thm]{Lemma}

\newtheorem{cor}[thm]{Corollary}

\newtheorem{conj}[thm]{Conjecture}

\theoremstyle{definition}

\newtheorem{defi}[thm]{Definition}

\newtheorem{conv}[thm]{Convention}

\newtheorem{rem}[thm]{Remark}

\newtheorem*{ack}{Acknowledgments}

 
\begin{document}
\bibliographystyle{amsalpha+}

\begin{abstract}
We use reduction maps to study the minimal model program.  Our main result is that the existence of a good minimal model for a klt pair $(X,\Delta)$ can be detected on a birational model of the base of the $(K_{X}+\Delta)$-trivial reduction map.  We then interpret the main conjectures of the minimal model program as a natural statement about the existence of curves on $X$.
\end{abstract}

 \maketitle

\tableofcontents

\section{Introduction}  The minimal model program relates the geometry of a complex projective variety $X$ to properties of its canonical divisor $K_{X}$.  One of the central ideas of the program is that $X$ should admit a birational model $X'$  where $K_{X'}$ has particularly close ties to geometry.  More precisely:

\begin{defi}
Let $(X,\Delta)$ be a $\mathbb{Q}$-factorial projective kawamata log terminal pair.  We say that $(X,\Delta)$ has a good minimal model if there is a sequence of $(K_{X} + \Delta)$-flips and divisorial contractions $\varphi: X \dashrightarrow X'$ such that some multiple of $K_{X'} + \varphi_{*}\Delta$ is basepoint free.
\end{defi}

The following conjecture, implicit in \cite{kawamata_pluri}, lies at the heart of the minimal model program:

\begin{conj} \label{existenceofgmm}
Let $(X,\Delta)$ be a $\mathbb{Q}$-factorial projective kawamata log terminal pair such that $K_{X} + \Delta$ is pseudo-effective.  Then $(X,\Delta)$ has a good minimal model.
\end{conj}


An important principle from \cite{kawamata_pluri} is that Conjecture \ref{existenceofgmm} can be naturally interpreted using numerical properties of $K_{X} + \Delta$.  Recall that the numerical dimension $\nu(D)$ of a pseudo-effective divisor $D$ as defined by \cite{N} and \cite{bdpp} is a numerical measure of  the ``positivity'' of $D$ (see Definition \ref{numerical Kodaira dimension}).   Conjecture \ref{existenceofgmm} is known in the two extremal cases: when $\nu(K_{X} + \Delta) = \dim X$ by \cite{BCHM} and when $\nu(K_{X} + \Delta) = 0$ by \cite{N} and \cite{d} (cf. \cite{g3}).  Furthermore, recent results of \cite{l} show that the existence of a good minimal model is equivalent to the equality $\kappa(K_{X} + \Delta) = \nu(K_{X} + \Delta)$.  

From this viewpoint, it is very natural to focus on morphisms $f: X \to Z$ for which $K_{X} + \Delta$ has good numerical behavior along the fibers.  Our main theorem shows that the existence of a good minimal model can be detected on (birational models of) the base of such maps.




\begin{thm}[(={Corollary \ref{main result}})]\label{main thm}
Let $(X,\Delta)$ be a $\mathbb{Q}$-factorial projective kawamata log terminal pair. Suppose that $f: X \to Z$ is a morphism with connected fibers to a normal projective variety $Z$ such that for a general fiber $F$ of $f$ we have $\nu((K_{X} + \Delta)|_{F}) = 0$. Then there exist a smooth projective birational model $Z'$ of $Z$ and a kawamata log terminal pair $(Z',\Delta_{Z'})$ such that $(X,\Delta)$ has a good minimal model if and only if $(Z',\Delta_{Z'})$ has a good minimal model.
\end{thm}

We use the techniques of \cite{ambro_nef} which proves the special case of Theorem \ref{main thm} when $K_{X} + \Delta$ is nef.  Our theorem has the important advantage that it can be applied to deduce the existence of a minimal model (rather than requiring that $K_{X } + \Delta$ be nef in the first place).  Some low dimension cases were worked out in \cite{bdpp}; related results appear in \cite{fukuda} and independently in the recent preprint \cite{siu}.

In order to apply Theorem \ref{main thm} in practice, the key question is whether, perhaps after a birational modification, one can find a map such that the numerical dimension of $(K_{X} + \Delta)|_{F}$ vanishes for a general fiber $F$.  The $(K_{X}+\Delta)$-trivial reduction map constructed in \cite{leh} satisfies precisely this property.  We develop a birational version of this theory better suited for working with the minimal model program. 


Finally, we reinterpret the existence of good minimal models as a statement about curves.  Recall that an irreducible curve $C$ is said to be movable if it is a member of a family of curves dominating $X$.   Movable curves are used to construct the $(K_{X} + \Delta)$-reduction map and are thus related to the existence of good minimal models by Theorem \ref{main thm}.   Conjecture \ref{existenceofgmm} implies the following well-known prediction:

\begin{conj} \label{vanishingintersectionconj}
Let $(X,\Delta)$ be a $\mathbb{Q}$-factorial projective kawamata log terminal pair.  Suppose that $(K_{X} + \Delta).C > 0$ for every movable curve $C$ on $X$.  Then $K_{X} + \Delta$ is big.
\end{conj}

We show that the two conjectures are equivalent: 

\begin{thm}[(=Corollary \ref{equivalenceofconjectures})]\label{leh_main}
Conjecture \ref{vanishingintersectionconj} holds up to dimension $n$ if and only if
Conjecture \ref{existenceofgmm} holds up to dimension $n$.
\end{thm}

The paper is organized as follows.  Section \ref{section2} is devoted to preliminary definitions and results. Section \ref{section3} develops the theory of the $D$-trivial reduction map to allow for applications to the minimal model program.  In Section \ref{section4}, we first discuss the relationship between abundance and the existence of good minimal models.  We then prove Theorem \ref{main thm} and Theorem \ref{leh_main}.

\begin{ack}The first author wishes to express his deep gratitude to his supervisor Professor Hiromichi Takagi for various comments and many suggestions. He would like to thank Professor Osamu Fujino and Doctor Ching-Jui Lai for answering his questions. He is partially supported by the Research Fellowships of the Japan Society for the Promotion of Science for Young Scientists. The second author would like to thank his advisor Professor James $\mathrm{M^{c}}$Kernan for his advice and support. He is supported under a National Science Foundation Graduate Research Fellowship.  The authors also thank the referee for the detailed suggestions.
\end{ack}

\section {Preliminaries}\label{section2}

In this section, we introduce definitions and collect some lemmas for the proof of main results.  

\begin{conv}
Throughout this paper we work over $\mathbb{C}$.  
\end{conv}

\subsection{Log pairs}

We start by discussing log pairs and their resolutions.

\begin{defi}
A log pair $(X,\Delta)$ consists of a normal variety $X$ and an effective $\mathbb Q$-Weil divisor $\Delta$ such that $K_X+\Delta$ is $\mathbb Q$-Cartier.  We say that $(X, \Delta)$ is {\em{kawamata log terminal}} 
if the discrepancy $a(E, X, \Delta) > -1$ for every prime divisor $E$ over $X$.
\end{defi}

\begin{defi}\label{log smooth}Let $(X,\Delta)$ be a kawamata log terminal pair and $\varphi:W \to X$ a log resolution of $(X,\Delta)$.  Choose $\Delta_{W}$ so that $$K_{W} + \Delta_{W} = \varphi^{*}(K_{X} + \Delta) + E$$ where $\Delta_{W}$ and $E$ are effective $\mathbb{Q}$-Weil divisors that have no common component.
We call $(W, \Delta_{W})$ a log smooth model of $(X,\Delta)$.  

Note that a minimal model of $(W,\Delta_{W})$ may not be a minimal model of $(X,\Delta)$.  To compensate for this deficiency, define for any $\epsilon > 0$
$$F=\sum_{F_i \text{ a }\varphi\text{-exceptional prime divisor}}F_i\ \text{and}\ \Delta_{W}^{\epsilon}=\Delta_{W}+\epsilon F.
$$
We call $(W, \Delta_{W}^{\epsilon})$ an $\epsilon$-{\em log smooth model}.  

\end{defi}

\begin{rem}\label{rem_log smooth} Note that our definition of a log smooth model differs from that of Birkar and Shokurov (cf. \cite{b2}).  Under our definition, for sufficiently small $\epsilon$ a good minimal model for $(W, \Delta_{W}^{\epsilon})$ is also a good minimal model for $(X,\Delta)$ (\cite[Lemma 3.6.10]{BCHM}).

\end{rem}

\subsection{$\mathbb{R}$-Cartier divisors}

We next turn to the birational theory of pseudo-effective $\mathbb{R}$-Cartier divisors.  

Suppose that $D=\sum _{j=1}^r d_j D_j$ is an $\mathbb{R}$-Weil divisor such that 
$D_j$ is a prime divisor for every $j$ and $D_i\ne D_j$ for $i\ne j$.  We define the 
{\em{round-down}} $\llcorner D\lrcorner =\sum _{j=1}^{r} \llcorner d_j \lrcorner D_j$,
where for every real number $x$, 
$\llcorner x\lrcorner$ 
is the integer defined by 
$x-1<\llcorner x\lrcorner \leq x$.  

\begin{defi}
For an $\mathbb{R}$-Cartier divisor $D$ on a normal projective variety $X$, the Iitaka dimension is
$$\kappa(D)=\mathrm{max}\{k \in \mathbb{Z}_{\geq 0}| \limsup_{m \to \infty}m^{-k} \mathrm{dim} \, H^0(X, \llcorner mD \lrcorner) >0  \}
$$ 
if $H^0(X, \llcorner mD \lrcorner) \not= 0$ for infinitely many $m \in \mathbb{N}$ or $\kappa(D)=-\infty$ otherwise.
\end{defi}

\begin{defi}[cf. \cite{elmnp}]\label{base locus}
Let $X$ be a normal projective variety and $D$ be an $\mathbb{R}$-Cartier divisor on $X$.  Fix an ample divisor $A$.  Define
$$\mathbf{B}(D)=\bigcap_{D \sim_{\mathbb{R}} E, E \geq 0} \mathrm{Supp}\,E,\ \text{and}\ \mathbf{B}_{-}(D)=\bigcup_{\epsilon \in \mathbb{R}_{>0}} \mathbf{B}(D+ \epsilon A).$$
As suggested by the notation $\mathbf{B}_{-}(D)$ is independent of the choice of $A$.
\end{defi}

As mentioned in the introduction, numerical properties of divisors play an important role this paper.  The essential technical tools we need were first introduced in \cite{N}: the $\sigma$-decomposition and the numerical dimension.

\begin{defi} Let $X$ be a smooth projective variety and $D$ be a pseudo-effective $\mathbb{R}$-Cartier divisor on $X$.  Fix an ample divisor $A$ on $X$.  Given a prime divisor $\Gamma$ on $X$, define
\begin{equation*}
\sigma_{\Gamma}(D) = \min \{\mult_{\Gamma}(D') | D' \geq 0 \textrm{ and } D' \sim_{\mathbb{Q}} D+\epsilon A \textrm{ for some } \epsilon>0 \}.
\end{equation*}
This definition is independent of the choice of $A$.

\end{defi}

It is shown in \cite{N} that for any pseudo-effective divisor $D$ there are only finitely many prime divisors $\Gamma$ such that $\sigma_{\Gamma}(D) > 0$.  Thus, the following definition make sense:

\begin{defi} Let $X$ be a smooth projective variety and $D$ be a pseudo-effective $\mathbb{R}$-Cartier divisor on $X$.  Define the $\mathbb{R}$-Cartier divisors $N_{\sigma}(D) = \sum_{\Gamma} \sigma_{\Gamma}(D) \Gamma$ and $P_{\sigma}(D) = D - N_{\sigma}(D)$.  The decomposition $$D = P_{\sigma}(D) + N_{\sigma}(D)$$ is known as the $\sigma$-{\em decomposition}.  It is also known as the {\em sectional decomposition} (\cite{kawamata_crepant}), the {\em divisorial Zariski decomposition} (\cite{bouck_dzd}), and the {\em numerical Zariski decomposition} (\cite{k-ab}).

\end{defi}

The basic properties of the $\sigma$-decomposition are:

\begin{lem}[\cite{N}] \label{pro-sigma-decomposition}Let $X$ be a smooth projective variety and $D$ a pseudo-effective $\mathbb{R}$-Cartier divisor.  Then
\begin{itemize}
\item[(1)] $\kappa(D) = \kappa(P_{\sigma}(D))$,
\item[(2)] $\Supp(N_{\sigma}(D)) \subset \mathbf{B}_{-}(D)$,
\item[(3)] for any prime divisor $\Gamma$ on $X$, $P_{\sigma}(D)|_{\Gamma}$ is pseudo-effective, and
\item[(4)] if $0 \leq D' \leq N_{\sigma}(D)$, then $D-D'$ is pseudo-effective and $N_{\sigma}(D-D') = N_{\sigma}(D) - D'$.
\end{itemize}

\end{lem}

\begin{proof} The proof of \cite[Theorem 5.5]{bouck_dzd} shows that the inclusion
$$i_{m}: H^{0}(X,\mathcal{O}_{X}(\lfloor mP_{\sigma}(D) \rfloor)) \to H^{0}(X,\mathcal{O}_{X}(\lfloor mD \rfloor))$$
is an isomorphism for any positive integer $m$.  This implies (1).
(2) is immediate from the definition. We see (4) by \cite[III, 1.8 Lemma]{N}. (3) follows from (4) and \cite[III, 1.14 Proposition (1)]{N}.
\end{proof}

Closely related to the $\sigma$-decomposition is the numerical dimension, a numerical measure of the positivity of a divisor.

\begin{defi}\label{numerical Kodaira dimension} Let $X$ be a normal projective variety, $D$ an $\mathbb{R}$-Cartier divisor and $A$ an ample divisor on $X$.  Set
$$\nu(D,A)=\mathrm{max}\{k \in \mathbb{Z}_{\geq 0}| \limsup_{m \to \infty}m^{-k} \mathrm{dim} \, H^0(X, \llcorner mD \lrcorner+A) >0  \}
$$ 
 if  $H^0(X, \llcorner mD \lrcorner+A) \not= 0$ for infinitely many $m \in \mathbb{N}$ or $\sigma(D,A)=- \infty$ otherwise.  Define
$$\nu(D) = \max_{A \textrm{ ample}} \nu(D,A).$$
\end{defi}

\begin{rem}\label{nakayama_rem} By the results of \cite{leh2}, this definition coincides with the notions of $\kappa_{\nu}(D)$ from \cite[V, 2.20, Definition]{N} and $\nu(D)$ from \cite[3.6, Definition]{bdpp}.

\end{rem}

\begin{lem}[{\cite[V, 2.7 Proposition]{N}, \cite[Theorem 6.7]{leh2}}] \label{pro-nu-dimension}Let $X$ be a normal projective variety, $D$ be an $\mathbb{R}$-Cartier divisor on $X$, and $\varphi: W \to X$ be a birational map from a normal projective variety $W$.   Then:
\begin{itemize}
\item[(1)] $\nu(D) = \nu(D')$ for any $\mathbb{R}$-Cartier divisor $D'$ such that $D' \equiv D$, 
\item[(2)] $\nu(\varphi^*D)= \nu(D)$,
\item[(3)] $\nu(D) \geq \kappa(D)$, and
\item[(4)] if $X$ is smooth then $\nu(D)= \nu(P_{\sigma}(D))$.
\end{itemize}
Furthermore, if $X$ is smooth then $\nu(D) = 0$ if and only if $P_{\sigma}(D) \equiv 0$.
\end{lem}

\subsection{Exceptional divisors}

Finally, we identify several different ways a divisor can be ``exceptional'' for a morphism.

\begin{defi}
Let $f: X \to Y$ be a morphism of normal projective varieties and let $D$ be an $\mathbb{R}$-Cartier divisor on $X$.  We say that $D$ is $f$-horizontal if $f(\Supp(D)) = Y$ or that $D$ is $f$-vertical otherwise.
\end{defi}

\begin{defi}[{\cite[III, Section 5.a]{N}}, {\cite[Definition 2.9]{l}} and {\cite[Definition 2.4]{tak2}}]\label{degenerate divisor}Let $f:X \to Y$ be a surjective morphism of normal projective varieties with connected fibers and $D$ be an effective $f$-vertical $\mathbb{R}$-Cartier divisor. We say that $D$ is $f$-{\em exceptional} if 
$$\mathrm{codim}\,f(\Supp\,D)\geq 2.$$
 
We call $D$ $f$-{\em degenerate} if for any prime divisor $P$ on $Y$ there is some prime divisor $\Gamma$ on $X$ such that $f(\Gamma) = P$ and $\Gamma \not \subset \Supp(D)$.  Note that every $f$-exceptional divisor is also $f$-degenerate.
\end{defi}

\begin{lem} \label{degenerate_lem}Let $f:X \to Y$ be a surjective morphism of normal projective varieties where $Y$ is $\mathbb{Q}$-factorial.  Suppose that $D$ is an effective $f$-vertical $\mathbb{R}$-Cartier divisor such that $f_{*}\mathcal{O}_{X}(\llcorner kD \lrcorner)^{**} \cong \mathcal{O}_{Y}$ for every positive integer $k$.  Then $D$ is $f$-degenerate.
\end{lem}

\begin{proof} If $D$ were not $f$-degenerate, there would be an effective $f$-exceptional divisor $E$ on $X$ and an effective $\mathbb{Q}$-Cartier divisor $T$ on $Y$ such that $f^{*}T \leq D+E$.  But since $E$ is $f$-exceptional we have $f_{*}\mathcal{O}_{X}(\llcorner k(D+E) \lrcorner)^{**} \cong \mathcal{O}_{Y}$ for every $k$, yielding a contradiction.
\end{proof}

Degenerate divisors behave well with respect to the $\sigma$-decomposition.

\begin{lem}[cf. {\cite[III.5.7 Proposition]{N}}]\label{negative_part_degenerate_lem}Let $f:X \to Y$ be a surjective morphism from a smooth projective variety to a normal projective variety and let $D$ be an effective $f$-degenerate divisor.  For any pseudo-effective $\mathbb{R}$-Cartier divisor $L$ on $Y$ we have $D \leq N_{\sigma}(f^{*}L + D)$ and $P_{\sigma}(f^{*}L + D) = P_{\sigma}(f^{*}L)$.
\end{lem}

\begin{proof} \cite[III.5.1 Lemma]{N} and \cite[III.5.2 Lemma]{N} together show that for an  $f$-degenerate divisor $D$ there is some component $\Gamma \subset \Supp(D)$ such that $D|_{\Gamma}$ is not pseudo-effective.  Since $P_{\sigma}(f^{*}L + D)|_{\Gamma}$ is pseudo-effective, we see that $\Gamma$ must occur in $N_{\sigma}(f^{*}L + D)$ with positive coefficient.

Set $D'$ to be the coefficient-wise minimum of the effective divisors $N_{\sigma}(f^{*}L+D)$ and $D$.  Since $D' \leq N_{\sigma}(f^{*}L + D)$, we may apply Lemma \ref{pro-sigma-decomposition} (4) to see that $$N_{\sigma}(f^{*}L + D) = N_{\sigma}(f^{*}L + D - D') + D'.$$  Suppose that $D' < D$.  Then $D - D'$ is still $f$-degenerate, so there is some component of $D-D'$ that appears in $N_{\sigma}(f^{*}L + (D-D')) = N_{\sigma}(f^{*}L + D) - D'$ with positive coefficient, a contradiction.  Thus $D = D' \leq N_{\sigma}(f^{*}L + D)$.  The final claim follows from Lemma \ref{pro-sigma-decomposition} (4).
\end{proof}



\section{Reduction maps and $\widetilde{\tau}(X,\Delta)$}\label{section3}  For a pseudo-effective divisor $D$ on a variety $X$, the $D$-trivial reduction map can be thought of as the ``quotient'' of $X$ by all movable curves $C$ satisfying $D.C = 0$.  A priori the $(K_{X} + \Delta)$-trivial reduction map may change if we pass to a log smooth model.  The main goal of this section is to develop a birational theory that takes this discrepancy into account.

\begin{thm}[{\cite[Theorem 1.1]{leh}}]\label{reduction}Let $X$ be a normal projective variety and $D$ be a pseudo-effective $\mathbb{R}$-Cartier divisor on $X$.  Then there exist a birational morphism $\varphi:W \to X$ from a smooth projective variety $W$ and a surjective morphism $f:W \to Y$ with connected fibers such that 
\begin{itemize}
\item[(1)] $\nu(\varphi^*D|_{F})=0$ for a general fiber $F$ of $f$,     
\item[(2)] if $w \in W$ is a very general point and $C$ is an irreducible curve through $w$ with $\dim\,f(C)=1$, then $\varphi^*D.C >0 $, and
\item[(3)] for any birational morphism $\varphi':W' \to X$ from a smooth projective variety $W'$ and dominant morphism $f':W' \to Y'$ with connected fibers satisfying condition (2), $f'$ factors birationally through $f$. 
\end{itemize} 
We call the composition $f \circ \varphi^{-1}: X \dashrightarrow Y$ the $D$-trivial reduction map.  Note that it is only unique up to birational equivalence.
\end{thm}

\begin{rem}\label{rem_reduction2} Property (2) is equivalent to the following: 
\begin{quote} (2') if $C$ is an irreducible movable curve with $\dim\,f(C)=1$, then $\varphi^*D.C >0 $.
\end{quote}
\end{rem}

\begin{rem}\label{rem_reduction}The $D$-trivial reduction map is different from the pseudo-effective reduction map (cf.~\cite{e2} and \cite{leh}), the partial nef reduction map (cf.~\cite{bdpp}), and Tsuji's numerically trivial fibration with minimal singular metrics (cf.~\cite{t} and \cite{e}).

\end{rem}

\begin{defi}\label{tau}Let $X$ be a normal projective variety and $D$ be a pseudo-effective $\mathbb{R}$-Cartier divisor on $X$. If $f: X \dashrightarrow Y$ denotes the $D$-trivial reduction map, we define 
$$\tau(D):=\dim\,Y.$$
\end{defi}

\begin{lem}\label{prop-tau-dimension}Let $X$ be a normal projective variety and $D$ be a pseudo-effective $\mathbb{R}$-Cartier divisor on $X$.  Then
\begin{itemize}
\item[(1)]  $\tau(D) \geq \nu(D) \geq \kappa(D)$, 
\item[(2)]  if $D'$ is a pseudo-effective $\mathbb{R}$-Cartier divisor on $X$ such that $D'\geq D$, then $\tau(D')\geq \tau(D)$, and
\item[(3)] $\tau(f^*D)=\tau(D)$ for every surjective morphism $f:Y \to X$ from a normal variety.
\end{itemize}
 
\end{lem}

\begin{proof}
Since $\kappa(D) \leq \nu(D)$ by Lemma \ref{pro-nu-dimension}, it suffices to prove the first inequality of (1).  Write $f: W \to Y$ for the $D$-trivial reduction map as in Theorem \ref{reduction}.  \cite[V, 2.22 Proposition]{N} states that $\nu(D) \leq \nu(D|_{F}) + \dim Y$ for a general fiber $F$ of $f$.  Since $\nu(D|_{F})=0$, we find $\nu(D) \leq \dim Y = \tau(D)$.  (2) and (3) follow easily from the definition.
\end{proof}

As mentioned above, a priori $\tau(K_{X} + \Delta)$ may change if we replace $(X,\Delta)$ by a log smooth model.  Thus we need to introduce a variant of this construction that accounts for every $\epsilon$-log smooth model simultaneously.

\begin{defi}\label{tilde_tau}Let $(X,\Delta)$ be a kawamata log terminal pair such that $K_X+\Delta$ is pseudo-effective. We define 
\begin{eqnarray*}
\widetilde{\tau}(X,\Delta)&=&\max\{ \, \tau(K_W+\Delta_{W}^{\epsilon}) \, | \, (W, \Delta_{W}^{\epsilon})\ \text{is an $\epsilon$-log smooth model}\\
& &\ \text{of}\ (X,\Delta) \text{ for some } \epsilon>0 \,  \}.
\end{eqnarray*}

\end{defi}

\begin{rem} \label{tilde_tauind}
Note that the maximum value of $\tau$ in the previous definition can be achieved by any sufficiently small $\epsilon > 0$.  
More precisely, suppose that $(X,\Delta)$ is a kamawata log terminal pair with $K_{X} + \Delta$ pseudo-effective and that $\varphi: W \to X$ is a log resolution of $(X,\Delta)$. Then the value of $\tau(K_{W} + \Delta_{W}^{\epsilon})$ for the $\epsilon$-log smooth model $(W,\Delta_{W}^{\epsilon})$ is independent of the choice of $\epsilon > 0$: if $C$ is a movable curve with $(K_{W} + \Delta_{W}^{\epsilon}).C = 0$ then (by the pseudo-effectiveness of $K_{X} + \Delta$ and \cite[0.2, Theorem]{bdpp}) we must have $\varphi^{*}(K_{X} + \Delta).C = 0$ and $E.C=0$ for any $\varphi$-exceptional divisor $E$.

\end{rem}

\begin{lem}\label{tildetau_lem}Let $(X,\Delta)$ be a projective kawamata log terminal pair such that $K_X+\Delta$ is pseudo-effective. For any $\epsilon > 0$, 
there exists an $\epsilon$-log smooth model $\varphi: (W,\Delta^{\epsilon}_W) \to (X,\Delta)$ such that the $(K_{W} + \Delta^{\epsilon}_W)$-trivial reduction map can be realized as a morphism on $W$ whose image has dimension $\widetilde{\tau}(X,\Delta)$.
\end{lem}

\begin{proof}
By Remark \ref{tilde_tauind}, the construction of the reduction map for any $\epsilon$-log smooth model is completely independent of the choice of $\epsilon > 0$.  So we may fix an arbitrary $\epsilon > 0$ for the remainder of the proof.

First choose an $\epsilon$-log smooth model $(W',\Delta^{\epsilon}_{W'})$ such that $\tau(K_{W'} + \Delta^{\epsilon}_{W'}) = \widetilde{\tau}(X, \Delta)$.  Let $(W,\Delta^{\epsilon}_{W})$ be an $\epsilon$-log smooth model such that there is a birational map $\varphi: W \to W'$ and a morphism $f: W \to Z$ resolving the $(K_{W'} + \Delta^{\epsilon}_{W'})$-trivial reduction map.
Note that there is some effective $\varphi$-exceptional divisor $E$ such that $K_{W} + \Delta^{\epsilon}_{W} + E \geq \varphi^{*}(K_{W'} + \Delta^{\epsilon}_{W'})$.  If $K_{W} + \Delta^{\epsilon}_{W}$ has vanishing intersection with a movable curve on $W$, then $E$ does as well, and hence so does $\varphi^{*}(K_{W'} + \Delta^{\epsilon}_{W'})$.  This means that $f$ factors birationally through the $(K_{W} + \Delta^{\epsilon}_{W})$-trivial reduction map by the universal property of reduction maps.   Since $\tau(K_{W'} + \Delta^{\epsilon}_{W'})$ is maximal over all $\epsilon$-log smooth models, $f$ must in fact be (birationally equivalent to) the $(K_{W} + \Delta^{\epsilon}_{W})$-trivial reduction map.
\end{proof}

\begin{rem}\label{rem_8author} If $D$ is a nef divisor, the $D$-trivial reduction map is birationally equivalent to the nef reduction map of $D$ (see \cite{8author}). Thus $n(D)=\tau(D)$, where $n(D)$ is the nef dimension of $D$ in \cite[Definition 2.7]{8author}. Moreover, for a projective kawamata log terminal pair $(X,\Delta)$ such that $K_X+\Delta$ is nef, $\tau(K_X+\Delta)=\widetilde{\tau}(X, \Delta)$ since the nef reduction map is almost holomorphic.
\end{rem}

The remainder of this section is devoted to proving that $\widetilde{\tau}(X,\Delta)$ is preserved upon passing to a minimal model.  In fact $\widetilde{\tau}$ does not change under any flip or divisorial contraction.

\begin{defi}\label{strong_fimily}Let $X$ be a normal projective variety and $T \subset \mathrm{Chow}(X)$ be an irreducible proper subvariety parametrizing $1$-cycles.  We say that the family of $1$-cycles $\{C_{t}\}_{t \in T}$ is a covering family if the map to $X$ is dominant.

Let $D$ be a $\mathbb{R}$-Cartier divisor on $X$. A covering family $\{C_t\}_{t \in T}$ is $D$-{\em trivial} if $D.C_t=0$ for all $t \in T$. A covering family $\{C_t\}_{t \in T}$ is {\em 1-connected} if for general points $x$ and $ y \in X$ there is $t \in T$ such that $C_t$ is an irreducible curve containing $x$ and $y$.
\end{defi}

\begin{prop}[cf. {\cite[Proposition 4.8]{leh}}]\label{triviality}Let $X$ be a normal projective variety and $D$ an $\mathbb{R}$-Cartier divisor on $X$. Suppose that there exists a $D$-trivial $1$-connected covering family $\{C_t\}_{t\in T}$. Then $\nu(D)=0$. 

\end{prop}

\begin{proof}For any birational map $\varphi: W \to X$, the strict transforms of the curves $C_{t}$ are still $1$-connecting.  Thus, the generic quotient (in the sense of \cite[Construction 3.2]{leh}) of $X$ by the family $\{ C_{t} \}_{t \in T}$ contracts $X$ to a point.  Thus $\nu(D)=\nu(f^*D)=0$ by \cite[Theorem 1.1]{leh}.
\end{proof}

\begin{prop}\label{pro-tau-dimension}Let $(X,\Delta)$ be a projective kawamata log terminal pair. Then $\nu(K_X+\Delta)=0$ if and only if there exists a $(K_X+\Delta)$-trivial $1$-connected covering family $\{C_t\}_{t\in T}$ such that $C_t \cap \mathbf{B}_{-}(K_X+\Delta)=\emptyset$ for general $t \in T$.
\end{prop}

\begin{proof} The reverse implication follows from Proposition \ref{triviality}.  Now assume that $\nu(K_X+\Delta)=0$. By \cite[Corollaire 3.4]{d}, there is a good minimal model $\varphi: X \dashrightarrow X'$ of $(X,\Delta)$ with $K_{X'} + \varphi_{*}\Delta \sim_{\mathbb{Q}} 0$.  Take a log resolution of $(X,\Delta)$ and $(X',\varphi_{*}\Delta)$:
\begin{equation*}
\xymatrix{ & W \ar[dl]_{p} \ar[dr]^{q}\\
 X \ar@{-->}[rr]&  & X'.}
\end{equation*} 
Set $E$ to be the effective $q$-exceptional divisor such that
$$p^*(K_X+\Delta)=q^*(K_{X'}+\varphi_{*}\Delta)+E.
$$
Now, since $K_{X'}+\varphi_{*}\Delta \sim_{\mathbb{Q}}0$,
$$p^*(K_X+\Delta)\sim_{\mathbb{Q}}E.
$$
Because $\mathrm{codim}\,q(\Supp\,E)\geq2$, there exists a complete intersection irreducible curve $C$ on $X'$ with respect to very ample divisors $H_1, \dots, H_{n-1}$ containing two general points $x,\,y$ such that $C\cap q(\Supp\,E)=\emptyset$. Let $\bar{C}$ be the strict transform of $C$ on $X$. Then
$$(K_X+\Delta).\bar{C}=0.
$$
In general, $B_{-}(D) \subseteq \mathrm{Supp} \, D$ for an effective $\mathbb{R}$-Cartier divisor $D$. Moreover, when $\nu(D)=0$, $B_{-}(D) = \mathrm{Supp}\, D$ by the equality $D=N_{\sigma}(D)$ and Lemma \ref{pro-sigma-decomposition} (2).  Thus we have $p(\Supp \, E) = \mathbf{B}_{-}(K_{X} + \Delta)$.   (See \cite[Theorem A (i)]{bbp} and \cite[Theorem 1.2]{cd} for more general results.)
The desired family can be constructed by taking the strict transform of deformations of $C$ which avoid $q(\Supp \, E)$.
\end{proof}

\begin{prop}\label{preserve}Let $(X,\Delta)$ be a $\mathbb{Q}$-factorial projective kawamata log terminal pair such that $K_X+\Delta$ is pseudo-effective. Suppose that 
$$\varphi:(X,\Delta) \dashrightarrow (X',\Delta')
$$
is a $(K_X + \Delta)$-flip or divisorial contraction. Then $\widetilde{\tau}(X,\Delta)=\widetilde{\tau}(X', \Delta')$.
\end{prop}

\begin{proof} Consider a log resolution of $(X,\Delta)$ and $(X',\Delta')$:
\begin{equation*}
\xymatrix{ & W \ar[dl]_{p} \ar[dr]^{q}\\
 X \ar@{-->}[rr]&  & X'.}
\end{equation*}
For a sufficiently small positive number $\epsilon$, write
$$K_W+\Delta^{\epsilon}_{W}=p^*(K_X+\Delta)+G$$
for  the $\epsilon$-log smooth structure induced by $(X,\Delta)$ and
$$K_W+\Delta'^{\epsilon}_{W}=q^*(K_{X'}+\Delta')+G',$$
for the $\epsilon$-log smooth structure induced by $(X',\Delta')$.  (Note that these structures might differ, if for example $\varphi$ is centered in a locus along which the discrepancy is negative.)  Using Lemma \ref{tildetau_lem}, we may assume that the log resolution $W$ satisfies
 \begin{itemize}
\item[(1)] the $(K_W+\Delta^{\epsilon}_{W})$-trivial reduction map is a morphism $f:W \to Y$ with $\dim\,Y=\widetilde{\tau}(X,\Delta)$, and
\item[(2)] the $(K_W+\Delta'^{\epsilon}_{W})$-trivial reduction map is a morphism $f':W \to Y'$ with $\dim\,Y'=\widetilde{\tau}(X',\Delta')$ and $Y'$ is smooth. 
 
 \end{itemize}
Since $\varphi$ is a $(K_{X} + \Delta)$-negative contraction, there is some effective $q$-exceptional divisor $E'$ such that $K_W+\Delta^{\epsilon}_{W}= K_W+\Delta'^{\epsilon}_{W} + E'$.  From Lemma \ref{prop-tau-dimension} (2), it holds that $\widetilde{\tau}(X,\Delta) \geq \widetilde{\tau}(X', \Delta')$.  Note that as $E'$ is $q$-exceptional and $(W,\Delta'^{\epsilon}_{W})$ is an $\epsilon$-log smooth model with $\epsilon>0$, we have $\mu E' \leq N_{\sigma}(K_{W} + \Delta'^{\epsilon}_{W})$ for some $\mu > 0$.  

Every movable curve $C$ with $(K_{W} + \Delta^{\epsilon}_{W}).C = 0$ also satisfies $(K_{W} + \Delta'^{\epsilon}_{W}).C = 0$.  The universal property of the $(K_{W} + \Delta^{\epsilon}_{W})$-reduction map implies that $f'$ factors birationally through $f$.  Conversely, by Proposition \ref{pro-tau-dimension} a very general fiber $F'$ of $f'$ admits a $1$-connecting covering family of $K_{W} + \Delta'^{\epsilon}_{W}$-trivial curves $\{C_{t}\}_{t \in T}$ such that $C_{t} \cap \mathbf{B}_{-}((K_{W} + \Delta'^{\epsilon}_{W})|_{F'
}) = \emptyset$ for general $t \in T$.
Since $E'|_{F'}$ is effective and $\nu((K_{W} + \Delta'^{\epsilon}_{W})|_{F'}) = 0$, we know that
$\mu E'|_{F'} \leq N_{\sigma}((K_{W} + \Delta'^{\epsilon}_{W})|_{F'})$.  Thus $E'.C_{t}=0$ for general $t$ since $C_{t}$ avoids $\mathbf{B}_{-}((K_{W} + \Delta'^{\epsilon}_W)|_{F'})$.  So
\begin{align*}
(K_W+\Delta^{\epsilon}_W).C_{t} & = (K_W+\Delta'^{\epsilon}_W + E').C_t\\
 &= 0.
\end{align*}
Since the $C_{t}$ form a $1$-connected covering family in a very general fiber of $f'$, the universal property of the $(K_{W} + \Delta'^{\epsilon}_{W})$-trivial reduction map implies that $f$ factors birationally through $f'$.  Thus $f$ and $f'$ are birationally equivalent.
\end{proof}

\begin{cor} \label{tauforgmm} Let $(X,\Delta)$ be a $\mathbb{Q}$-factorial projective kawamata log terminal pair such that $K_X+\Delta$ is pseudo-effective.  Suppose that $(X,\Delta)$ has a good minimal model.  Then $$\widetilde{\tau}(X,\Delta) = \tau(K_{X} + \Delta) = \kappa(K_{X} + \Delta).$$
\end{cor}

\begin{proof}
We always have $\widetilde{\tau}(X,\Delta) \geq \tau(K_{X} + \Delta) \geq \kappa(K_{X} + \Delta)$.  Since $\widetilde{\tau}(X,\Delta)$ is preserved by steps of the minimal model program, the equality of the outer two quantities can be checked on the good minimal model.

So suppose that $K_{X} + \Delta$ is semiample and let $f: X \to Z$ denote the morphism defined by a sufficiently high multiple of $K_{X} + \Delta$.  For any birational map $\varphi: W \to X$, we can intersect a general fiber $F$ of $f$ with general very ample divisors to find a movable curve $C$ contained in $F$ and avoiding the image of the $\varphi$-exceptional locus.  In particular, if $(W,\Delta^{\epsilon}_{W}$) is an $\epsilon$-log smooth model  of $(X,\Delta)$, the strict transform $\widetilde{C}$ of $C$ satisfies $(K_{W} + \Delta^{\epsilon}_{W}).\widetilde{C} = 0$ and the $(K_{W} + \Delta^{\epsilon}_{W})$-trivial reduction map is $f \circ \varphi$.  This shows that $\tau(K_{W} + \Delta^{\epsilon}_{W}) = \kappa(K_{X} + \Delta)$ for every $\epsilon$-log smooth model $(W,\Delta^{\epsilon}_{W})$.
\end{proof}

\section{Applications to the minimal model program}\label{section4}
In this section we first discuss how the existence of a good minimal model can be reinterpreted using the notion of abundance.  We then prove Lemma \ref{inductivestep}, the main technical tool, and conclude with proofs of the theorems.

\subsection{Abundance and the existence of good minimal models}

The notion of abundance was introduced to capture those divisors with particularly good numerical behavior.

\begin{lem}[{\cite[V.4.2 Corollary]{N}}]\label{abundantequivalence}
Let $(X,\Delta)$ be a projective kawamata log terminal pair such that $K_{X} + \Delta$ is pseudo-effective. Then the following are equivalent:
\begin{enumerate}
\item $\kappa(K_X+\Delta) = \nu(K_X+\Delta)$.
\item  $\kappa(K_X+\Delta) \geq 0$ and if $\varphi: X' \to X$ is a birational morphism and $f: X' \to Z'$ a morphism resolving the Iitaka fibration for $K_X+\Delta$, then
\begin{equation*}
\nu(\varphi^{*}(K_X+\Delta)|_{F}) = 0
\end{equation*}
for a general fiber $F$ of $f$.
\end{enumerate}
If either of these equivalent conditions hold, we say that $K_{X} + \Delta$ is abundant.
\end{lem}

To relate abundance to the existence of minimal models, we will use the following special case.

\begin{thm}[{\cite[V.4.9 Corollary]{N}} and {\cite[Corollaire 3.4]{d}} (cf. \cite{g3})] \label{numdim0case} Let $(X,\Delta)$ be a $\mathbb{Q}$-factorial projective kawamata log terminal pair such that $\nu(K_{X} + \Delta) = 0$.  Then $K_{X} + \Delta$ is abundant and $(X,\Delta)$ admits a good minimal model.

\end{thm}

The following theorem is known to experts; for example, see \cite[Remark 2.6]{dhp}.  The theorem is a consequence of \cite[Theorem 4.4]{l}.  Note that the statement does not involve any inductive assumptions.

\begin{thm}[cf. \cite{dhp}] \label{goodminimalmodel}
Let $(X,\Delta)$ be a $\mathbb{Q}$-factorial projective kawamata log terminal pair. Then $K_{X} + \Delta$ is abundant if and only if $(X,\Delta)$ has a good minimal model.
\end{thm}

\begin{proof}First suppose that $(X,\Delta)$ has a good minimal model $(X',\Delta')$.  Let $Y$ be a common resolution of $X$ and $X'$ (with morphisms $f$ and $g$ respectively) and write
\begin{equation*}
f^{*}(K_{X} + \Delta) = g^{*}(K_{X'} + \Delta') + E
\end{equation*}
where $E$ is an effective $g$-exceptional $\mathbb{Q}$-divisor.  Thus $$P_{\sigma}(f^{*}(K_{X} + \Delta)) = P_{\sigma}(g^{*}(K_{X'} + \Delta'))$$ and since the latter divisor is semi-ample, the first is semi-ample as well.  The abundance of $K_{X} + \Delta$ follows from the fact that the Iitaka and numerical dimensions are invariant under pulling-back and passing to the positive part.

Conversely, suppose that $K_{X} + \Delta$ is abundant. Let $f:(X,\Delta) \dashrightarrow Z$ be the Iitaka fibration of $K_X+\Delta$.  Choose an $\epsilon$-log smooth model $\varphi:(W,\Delta^{\epsilon}_W) \to X$ with sufficiently small $\epsilon>0$ so that $f$ is resolved on $W$.
By \cite[Lemma 3.6.10]{BCHM} we can find a minimal model for $(X,\Delta)$ by constructing a minimal model of $(W,\Delta^{\epsilon}_W)$ . Moreover we see that $f\circ\varphi$ is also the Iitaka fibration of $K_W+\Delta^{\epsilon}_W$ and $\nu(K_{W} + \Delta^{\epsilon}_W ) = \nu(K_{X} + \Delta)$.  Replacing $(X,\Delta)$ by $(W,\Delta^{\epsilon}_{W})$, we may suppose that the Iitaka fibration $f$ is a morphism on $X$.

By \cite[V.4.2 Corollary]{N}, $\nu(K_F+\Delta_F)=0$ where $F$ is a general fiber of $f$ and $K_F+\Delta_F=(K_X+\Delta)|_{F}$.  Thus $(F,\Delta_F)$ has a good minimal model by Theorem \ref{numdim0case}.
The arguments of \cite[Theorem 4.4]{l} for $(X,\Delta)$ now show that $(X,\Delta)$ has a good minimal model. 
\end{proof}

\subsection{Main results}

The following lemma is key for proving our main results.

\begin{lem} \label{inductivestep}
Let $(X,\Delta)$ be a projective kawamata log terminal pair.  Suppose that $f: X \to Z$ is a projective morphism with connected fibers to a projective normal variety $Z$ such that $\nu((K_{X} + \Delta)|_{F}) = 0$ for a general fiber $F$ of $f$.  Then there exists a log resolution $\mu: X' \to X$ of $(X,\Delta)$, a projective smooth birational model $Z'$ of $Z$, a kawamata log terminal pair $(Z',\Delta_{Z'})$, and a morphism $f': X' \to Z'$ birationally equivalent to $f$ such that
\begin{equation*}
P_{\sigma}(\mu^{*}(K_{X} + \Delta)) \sim_{\mathbb{Q}} P_{\sigma}(f'^{*}(K_{Z'} + \Delta_{Z'})).
\end{equation*}
\end{lem}

\begin{proof}
We first reduce to the case when $f$ satisfies the stronger property that $(K_{X} + \Delta)|_{F} \sim_{\mathbb{Q}} 0$ for a general fiber $F$ of $f$.  Run the relative minimal model program with scaling of an ample divisor on $(X,\Delta)$ over $Z$.  By \cite[Theorem 2.3]{fujino-ss}, after finitely many steps we obtain a birational model $\psi: X \dashrightarrow X_{i}$ with a morphism $f_{i}: X_{i} \to Z$ such that $\mathbf{B}_{-}((K_{X_{i}} + \psi_{*}\Delta)|_{F_{i}})$ has no divisorial components for a general fiber $F_{i}$ of $f_{i}$.  Moreover  $\nu((K_{X_{i}} + \psi_{*}\Delta)|_{F_{i}}) = \kappa((K_{X_{i}} + \psi_{*}\Delta)|_{F_{i}})= 0$ by Theorem \ref{numdim0case}.   Thus we have the property $(K_{X_{i}} + \psi_{*}\Delta)|_{F_{i}} \sim_{\mathbb{Q}} 0$ (see the proof of Proposition \ref{pro-tau-dimension}).

Furthermore, recall that for any common log resolution $W$ of $(X,\Delta)$ and $(X_{i},\psi_{*}\Delta)$
\begin{equation*}
\xymatrix{ & W \ar[dl]_{p} \ar[dr]^{q}\\
 X \ar@{-->}[rr]^{\psi} &  & X_{i}.}
\end{equation*}
there is an effective $q$-exceptional divisor $M$ on $W$ such that $p^{*}(K_{X} + \Delta) = q^{*}(K_{X_{i}} + \psi_{*}\Delta) + M$.  In particular $P_{\sigma}(p^{*}(K_{X} + \Delta)) = P_{\sigma}(q^{*}(K_{X_{i}} + \psi_{*}\Delta))$ by Lemma \ref{negative_part_degenerate_lem}.  If we prove the statement of the theorem for $(X_{i},\psi_{*}\Delta)$ for the morphism $f: X' \to Z'$, the conclusion also holds for any composition $f \circ q': W \to X' \to Z'$ where $W$ is a smooth birational model of $X'$.  Letting $W$ be a common log resolution as above, we conclude the statement of the theorem for $(X,\Delta)$.

So we assume from now on that $(K_{X} + \Delta)|_{F} \sim_{\mathbb{Q}} 0$ for a general fiber $F$ of $f$. We may apply the techniques of \cite[4.4]{fm} to find a morphism birationally equivalent to $f$ that satisfies nice properties.  That is, there exist:
\begin{itemize}
\item a log smooth model $(X',\Delta')$ of $(X,\Delta)$ with birational map $\mu: X' \to X$, where we write $K_{X'} + \Delta' = \mu^{*}(K_{X} + \Delta) + E$ for an effective $\mu$-exceptional divisor $E$,
\item a $\mathbb{Q}$-Cartier divisor $B$ on $X'$ which we express as the difference $B = B^{+} - B^{-}$ of effective divisors $B^{+}$ and $B^{-}$ with no common components,
\item a smooth variety $Z'$ and a divisor $\Delta_{Z'}$, and
\item a morphism $f': X' \to Z'$ birationally equivalent to $f$
\end{itemize}
satisfying the properties:
\begin{itemize}
\item[(1)] $K_{X'}+\Delta' \sim_{\mathbb{Q}} f'^{*}(K_{Z'}+\Delta_{Z'})+B$,
\item[(2)] there is a positive integer $b$ such that  
$$H^0(X',mb(K_{X'}+\Delta'))=H^0(Z',mb(K_{Z'}+\Delta_{Z'}))
$$
for any positive integer $m$,
\item[(3)] $B^{-}$ is $f'$-exceptional and $\mu$-exceptional, and
\item[(4)] $f'_*\mathcal{O}_{X'}(\llcorner lB^{+} \lrcorner)=\mathcal{O}_{Z'}$ for every positive integer $l$.
\end{itemize}
We next apply the results of \cite{ambro}.  Choose an integer $m$ so that $m(K_{X} + \Delta)$ is a Cartier divisor and $m(K_{X} + \Delta)|_{F} \sim 0$ for the general fiber $F$ of $f$.  Then $f_{*}\mathcal{O}_{X}(m(K_{X} + \Delta)) \neq 0$ since it is invertible over an open subset of $Z$ by Grauert's theorem.  Thus there is an ample divisor $A$ on $Z$ so that $H^{0}(X,m(K_{X} + \Delta) + f^{*}A) \neq 0$.  Choose an effective divisor $\Omega$ in this linear system.  $\Omega$ must be $f$-vertical since $\Omega|_{F} \equiv 0$ for the general fiber $F$ of $f$.  In the notation of \cite{ambro}, $f: (X,\Delta - \frac{1}{m}\Omega) \to Z$ is an LC-trivial fibration. \cite[Theorem 3.3]{ambro} allows us to conclude the additional key property that $(Z',\Delta_{Z'})$ may be taken to be a kawamata log terminal pair (perhaps after additional birational modifications which we absorb into the notation).

We conclude by comparing $P_{\sigma}(K_{X} + \Delta)$ and $P_{\sigma}(K_{Z'} + \Delta_{Z'})$.  Write $B^+ = B^+_h + B^+_v$ for the decomposition into the $f'$-horizontal components $B^+_h$ and the $f'$-vertical components $B^+_v$.  We analyze in turn $B^-$, then $B^+_h$, then $B^+_v$.

First consider $B^{-}$.  Since $B^{-}$ and $E$ are $\mu$-exceptional, Lemma \ref{negative_part_degenerate_lem} shows that $E + B^{-} \leq N_{\sigma}(\mu^{*}(K_{X} + \Delta) + E + B^-)$.  Thus we may apply Lemma \ref{pro-sigma-decomposition} (4) to $B^{-}$ to obtain
\begin{align*} 
N_{\sigma}(K_{X'} + \Delta' + B^{-}) &=N_{\sigma}(\mu^{*}(K_{X} + \Delta) + E + B^-)\\
& =  N_{\sigma}(\mu^{*}(K_{X} + \Delta) + E) + B^{-} \text{\ by Lemma \ref{pro-sigma-decomposition} (4),}\\
& =  N_{\sigma}(K_{X'} + \Delta') + B^{-}. \tag{$\ast$}\label{siki}
\end{align*}

Next consider $B^+_h$.  Note that $\nu((K_{X'}+\Delta')|_{F'})=0$ for a general fiber $F'$ of $f'$.  Indeed, for a general $F'$ the map $\mu|_{F'}$ is birational so that $E|_{F'}$ is $\mu|_{F'}$-exceptional.  By Lemma \ref{negative_part_degenerate_lem} $P_{\sigma}((K_{X'} + \Delta')|_{F'}) = P_{\sigma}(\mu^{*}(K_{X} + \Delta)|_{F'})$ and so by Lemma \ref{pro-nu-dimension}
$$\nu((K_{X'} + \Delta')|_{F}) = \nu(\mu^{*}(K_{X} + \Delta)|_{F'}) = 0.$$
This implies that if $D$ is an effective divisor on $F'$ such that $(K_{X'} + \Delta')|_{F'} - D$ is pseudo-effective then $D \leq N_{\sigma}((K_{X'} + \Delta')|_{F'})$.  Applying this fact to $B_{h}^{+}|_{F'}$ we obtain
$$B_{h}^+|_{F'} \leq N_{\sigma}((K_{X'}+\Delta')|_{F'}) \leq N_{\sigma}(K_{X'}+\Delta')|_{F'}.$$
As $F'$ is general, this equation implies that $B^+_h \leq N_{\sigma}(K_{X'}+\Delta')$.  By our earlier work for $B^{-}$, it is also true that $B^+_h \leq N_{\sigma}(K_{X'} + \Delta' + B^{-})$.

Finally, consider $B^+_v$.  By Lemma \ref{degenerate_lem}, property (4) shows that $B^+_v$ is $f'$-degenerate.  Again applying Lemma \ref{negative_part_degenerate_lem},
\begin{align*}
P_{\sigma}(f'^{*}(K_{Z'} + \Delta_{Z'}) + B^{+}_{v}) = P_{\sigma}(f'^{*}(K_{Z'} + \Delta_{Z'})).\tag{$\ast \ast$}\label{siki2}
\end{align*}
Putting it all together, we find
\begin{align*}
P_{\sigma}(\mu^{*}(K_{X} + \Delta)) & =  P_{\sigma}(K_{X'} + \Delta') \text{\ since $E$ is $\mu$-exceptional,}\\
& = P_{\sigma}(K_{X'} + \Delta' + B^{-}) \text{\ by (\ref{siki}),}\\
&= P_{\sigma}(K_{X'} + \Delta' + B^{-} - B^{+}_{h})\text{\ by analysis of } B^{+}_{h},\\
&\sim_{\mathbb{Q}}P_{\sigma}(f'^{*}(K_{Z'} + \Delta_{Z'}) + B^{+}_{v} 
) \text{\ by property (1),}\\
&= P_{\sigma}(f'^{*}(K_{Z'} + \Delta_{Z'})) \text{\ by (\ref{siki2})}.
\end{align*}
\end{proof}

\begin{cor}\label{main result}
Let $(X,\Delta)$ be a projective $\mathbb{Q}$-factorial kawamata log terminal pair. Suppose that $f: X \to Z$ is a morphism with connected fibers to a projective normal variety $Z$ such that $\nu((K_{X} + \Delta)|_{F}) = 0$ for a general fiber $F$ of $f$. Then there exists a smooth projective birational model $Z'$ of $Z$ and a kawamata log terminal pair $(Z',\Delta_{Z'})$ such that $(X,\Delta)$ has a good minimal model if and only if $(Z',\Delta_{Z'})$ has a good minimal model.
\end{cor}

\begin{proof}By Lemma \ref{inductivestep} and the fact that the Iitaka and numerical dimensions are invariant under pull-back and under passing to the positive part $P_{\sigma}$, we see that $K_X + \Delta$ is abundant if and only if $K_{Z'} + \Delta_{Z'}$ is abundant.  We conclude by Theorem \ref{goodminimalmodel}.
\end{proof}

\begin{thm}\label{mm}Assume the existence of good minimal models for $\mathbb{Q}$-factorial projective kawamata log terminal pairs in dimension $d$. Let $(X,\Delta)$ be a $\mathbb{Q}$-factorial projective kawamata log terminal pair such that $K_X+\Delta$ is pseudo-effective and $\widetilde{\tau}(X,\Delta)= d$.  Then there exists a good log minimal model of $(X,\Delta)$. 

\end{thm}

\begin{proof} Using Lemma \ref{tildetau_lem}, we can find a birational morphism $\varphi:W \to X$ from an $\epsilon$-log smooth model $(W,\Delta^{\epsilon}_{W})$ of $(X,\Delta)$ for a sufficiently small positive number $\epsilon$ and a  morphism $f:W \to Z$ with connected fibers such that 

\begin{itemize}
\item[(i)] $\nu((K_{W} + \Delta^{\epsilon}_{W})|_{F}) = 0$ for the general fiber $F$ of $f$ and
\item[(ii)] $\dim\,Z=\widetilde{\tau}(X,\Delta)$.
\end{itemize}
Corollary \ref{main result} implies that $(W,\Delta^{\epsilon}_{W})$ has a good minimal model.  $(X,\Delta)$ then has a good minimal model by \cite[Lemma 3.6.10]{BCHM}.
\end{proof}

\begin{cor} \label{equivalenceofconjectures}
Conjecture \ref{vanishingintersectionconj} holds up to dimension $n$ if and only
Conjecture \ref{existenceofgmm} holds up to dimension $n$.
\end{cor}

\begin{proof}Assume that Conjecture \ref{vanishingintersectionconj} holds up to dimension $n$. By induction on dimension, we may assume that Conjecture \ref{existenceofgmm} holds up to dimension $n-1$. Let $(X,\Delta)$ be a projective kawamata log terminal pair of dimension $n$.  If $\widetilde{\tau}(X,\Delta)<n$ then $K_X+\Delta$ is abundant by Theorem \ref{mm} and the induction hypothesis.  If $\widetilde{\tau}(X,\Delta)=n$ then some $\epsilon$-log smooth model $(W,\Delta^{\epsilon}_{W})$ does not admit a $(K_{W} + \Delta^{\epsilon}_{W})$-trivial covering family of curves.  Since we are assuming Conjecture \ref{vanishingintersectionconj}, $K_{W} + \Delta^{\epsilon}_{W}$ must be big.  \cite{BCHM} then gives the existence of a good minimal model for $(W,\Delta^{\epsilon}_{W})$ and hence also for $(X,\Delta)$.

Conversely, assume that Conjecture \ref{existenceofgmm} holds up to dimension $n$.  Suppose that $(X,\Delta)$ is a projective kawamata log terminal pair of dimension at most $n$ admitting a $(K_{X} + \Delta)$-trivial covering family of curves.  Then $\tau(K_{X} + \Delta) < \dim X$.  Corollary \ref{tauforgmm} shows that $\kappa(K_{X} + \Delta) < \dim X$ so $K_{X} + \Delta$ is not big.
\end{proof}

\begin{rem} \label{trivialreductionremark}
It seems likely that one could formulate a stronger version of Corollary \ref{equivalenceofconjectures} using the pseudo-effective reduction map for $K_{X} + \Delta$ (cf. \cite{e2} and \cite{leh}).  The difficulty is that the pseudo-effective reduction map only satisfies the weaker condition $\nu(P_{\sigma}(K_{X} + \Delta)|_{F}) = 0$ on a general fiber $F$, so it is unclear how to use the inductive hypothesis to relate $F$ with $X$.
\end{rem}

\end{document}